\documentclass[12pt]{article}
\usepackage[a4paper,margin=1.2in]{geometry}
\usepackage{amssymb}
\usepackage{amsthm}
\usepackage{amsmath}
\usepackage{amsfonts}
\usepackage{graphicx}
\usepackage{float}
\usepackage[all]{xy}
\usepackage{scalerel}
\usepackage{mathtools}
\usepackage{diagbox}
\usepackage{stmaryrd}
\usepackage{hyperref}
\hypersetup{
     colorlinks=true,
     linkcolor=red,
     filecolor=blue,
     citecolor = black,      
     urlcolor=blue,
     }

\def\exp{\mathop{\mbox{\textsf{exp}}}\nolimits}

\newcommand{\genstirlingI}[3]{%
  \genfrac{[}{]}{0pt}{#1}{#2}{#3}%
}
\newcommand{\genstirlingII}[3]{%
  \genfrac{\{}{\}}{0pt}{#1}{#2}{#3}%
}
\newcommand{\genauto}[3]{%
  \genfrac{\llbracket}{\rrbracket}{0pt}{#1}{#2}{#3}%
}
\newcommand{\stirlingfk}[2]{\genstirlingI{}{#1}{#2}}
\newcommand{\stirlingsk}[2]{\genstirlingII{}{#1}{#2}}
\newcommand{\cfauto}[2]{\genauto{}{#1}{#2}}

\newcommand{\C}{\mathbb{C}}

\newcommand{\R}{\mathbb{R}}

\newcommand{\x}{\mathbf{x}}
\newcommand{\y}{\mathbf{y}}

\newcommand{\gr}{\textsf{gr}}

\newcommand{\seqnum}[1]{\href{https://oeis.org/#1}{\rm\underline{#1}}}

\DeclareMathOperator{\sech}{sech}
\DeclareMathOperator{\arctanh}{arctanh}

\theoremstyle{plain}
\newtheorem{theorem}{Theorem}

\theoremstyle{definition}
\newtheorem{definition}[theorem]{Definition}
\newtheorem{proposition}{Proposition}
\newtheorem{corollary}{Corollary}

\newtheorem{conjecture}{Conjecture}

\begin{document}
%\title{\textbf{Solution of the Differential Equation $y^{(k)}=e^{ay}$, Special Values of Bell Polynomials and  $(k,a)$-Autonomous Coefficients}}
%\author{Ronald Orozco L\'opez}

\begin{center}
\vskip 1cm{\LARGE\bf Solution of the Differential Equation $y^{(k)}=e^{ay}$, Special Values of Bell Polynomials and $(k,a)$-Autonomous Coefficients}
\vskip 1cm
\Large
Ronald Orozco López\\
Departamento de Matemáticas \\
Universidad de los Andes\\
Bogotá, 111711\\
Colombia \\
\href{mailto:rj.orozco@uniandes.edu.co}{\tt rj.orozco@uniandes.edu.co} \\
\ \\
\end{center}

\newcommand{\Addresses}{{% additional braces for segregating \footnotesize
%  \bigskip
%  \footnotesize

%  R.~Orozco, \textsc{Departamento de Matematicas, Universidad de los Andes, 
%   Carrera 1 N. 18A-12 Bogot\'a, Colombia}\par\nopagebreak
%\textit{E-mail address}, R.~Orozco: \texttt{rj.orozco@uniandes.edu.co}
  
}}

%\maketitle

\begin{abstract}
In this paper special values of Bell polynomials are given by using the power series solution of the equation $y^{(k)}=e^{ay}$. In addition, complete and partial exponential autonomous functions, exponential autonomous polynomials, autonomous polynomials and $(k,a)$-autonomous coefficients are defined. Finally, we show the relationship between various numbers counting combinatorial objects and binomial coefficients, Stirling numbers of second kind and autonomous coefficients.
\end{abstract}
%{\bf Keywords:} \\
% {\bf Mathematics Subject Classification 2020:} 

\section{Introduction}
It is a known fact that Bell polynomials are closely related to derivatives of composition of functions. For example, Faa di Bruno \cite{Faa}, Foissy \cite{Foissy}, and Riordan \cite{Riordan_j} showed that Bell polynomials are a very useful tool in mathematics to represent the $n$-th derivative of the composition of functions. Also, Bernardini and Ricci \cite{Bernardini}, Yildiz et al. \cite{Yildiz}, Caley \cite{Caley}, and Wang \cite{Wang} showed the relationship between Bell polynomials and differential equations. On the other hand, Orozco \cite{Orozco} studied the convergence of the analytic solution of the autonomous differential equation $y^{(k)}=f(y)$ by using Faa di Bruno's formula. We can then look at differential equations as a source for investigating special values of Bell polynomials.

In this paper we will focus on finding special values of Bell polynomials when the vector field $f(x)$ of the autonomous differential equation $y^{(k)}=f(y)$ is the exponential function. We will not consider the convergence of the solutions, but we will show that well known numbers such as reduced tangent numbers, Bernoulli numbers, Euler zigzag numbers, Blasius numbers, among others, can be constructed using Bell polynomials. In general, a special class of numbers, which have not yet been studied, are constructed using Bell polynomials. On the other hand, a new family of numbers, called $(k,a)$-autonomous coefficients, is obtained for each value of $k$. Four conjectures about these numbers are established.

This paper is divided as follows. We begin with a summary of results on complete and partial Bell polynomials, which will be used to demonstrate the main results presented here. Next, we introduce the complete and partial exponential autonomous functions, the recurrence relations of these are constructed using Bell polynomials, and some recurrence relations of solutions of various initial value problems are given. In the fourth section the $(k,a)$-autonomous coefficients are introduced. From these numbers we can obtain the triangular numbers, the 8-sequence numbers of $[1,n]$ with 2 contiguous pairs, among others. We finish this work by studying the cases $k=2,3,4$ for the autonomous differential equation $y^{(k)}=e^{ay}$. 

\section{Bell exponential polynomials}

The following basic results can be found at Comtent \cite{Comtent}, and Riordan \cite{Riordan_b}. Exponential Bell polynomials are used to encode information on the ways in which a set can be partitioned, hence they are a very useful tool in combinatorial analysis. Bell polynomials are obtained from the derivatives of composite functions and are given by Faa Di Bruno's formula \cite{Faa}. Bell \cite{Bell}, Gould \cite{Gould_Q} and Mihoubi \cite{Mihoubi} provided important results on these polynomials. We start with the definition of partial Bell polynomials
\begin{definition}
The exponential partial Bell polynomials are the polynomials
\begin{equation*}
    B_{n,k}(x_{1},x_{n},...,x_{n-k+1})
\end{equation*}
in the variables $x_{1},x_{2},...$ defined by the series expansion
\begin{equation}\label{eqn_egf}
\exp\left(u\sum_{j=1}^{\infty}x_{j}\frac{t^{j}}{j!}\right)=1+\sum_{n=1}^{\infty}\frac{t^{n}}{n!}\sum_{k=1}^{n}u^{k}B_{n,k}(x_{1},x_{2},...,x_{n-k+1}).
\end{equation}
\end{definition}
The following result gives the explicit way to calculate the partial Bell polynomials
\begin{theorem}
The partial or incomplete exponential Bell polynomials are given by
\begin{equation*}
B_{n,k}(x_{1},\ldots,x_{n-k+1})=\sum\frac{n!}{c_{1}!c_{2}!\cdots c_{n-k+1}!}\left(\frac{x_{1}}{1!}\right)^{c_{1}}\cdots\left(\frac{x_{n-k+1}}{(n-k+1)!}\right)^{c_{n-k+1}}
\end{equation*}
where the summation takes place over all integers $c_{1},c_{2},...,c_{n-k+1}\geq0$, such that
\begin{eqnarray*}
c_{1}+2c_{2}+\cdots+(n-k+1)c_{n-k+1}&=&n\\
c_{1}+c_{2}+\cdots+c_{n-k+1}&=&k
\end{eqnarray*}
\end{theorem}

The following are special cases of partial Bell polynomials and will be very useful for proving results in this paper
\begin{align}
B_{n,1}(x_{1},...,x_{n})&=x_{n},\\
B_{n,2}(x_{1},...,x_{n-1})&=\frac{1}{2}\sum_{k=1}^{n-1}\binom{n}{k}x_{k}x_{n-k},\\
B_{n,n-a}(x_{1},...,x_{a+1})&=\sum_{j=a+1}^{2a}\frac{j!}{a!}\binom{n}{j}x_{1}^{n-j}B_{a,j-a}\left(\frac{x_{2}}{2},\cdots,\frac{x_{2(a+1)-j}}{2(a+1)-j}\right)\label{eqn_n,n-a},\nonumber\\
&\ 1\leq a< n,
\end{align}
\begin{align}
B_{n,n}(x_{1})&=x_{1}^{n},\label{eqn_n_n}\\
B_{n,n-1}(x_{1},x_{2})&=\binom{n}{2}x_{1}^{n-2}x_{2},\label{eqn_n_n-1}\\
B_{n,n-2}(x_{1},x_{2},x_{3})&=\binom{n}{3}x_{1}^{n-3}x_{3}+3\binom{n}{4}x_{1}^{n-4}x_{2}^{2},\label{eqn_n_n-2}\\
B_{n,n-3}(x_{1},x_{2},x_{3},x_{4})&=\binom{n}{4}x_{1}^{n-4}x_{4}+10\binom{n}{5}x_{1}^{n-5}x_{2}x_{3}+15\binom{n}{6}x_{1}^{n-6}x_{2}^{3},\label{eqn_n_n-3}\\
B_{n,n-4}(x_{1},x_{2},x_{3},x_{4},x_{5})&=\binom{n}{5}x_{1}^{n-5}x_{5}+5\binom{n}{6}x_{1}^{n-6}[3x_{2}x_{4}+2x_{3}^{2}]\nonumber\\
&+105\binom{n}{7}x_{1}^{n-7}x_{2}^{2}x_{3}+105\binom{n}{8}x_{1}^{n-8}x_{2}^{4}.\label{eqn_n,n-4}
\end{align}

Some values of partial Bell polynomials are
\begin{align*}
    B_{n,k}(0!,1!,...,(n-k)!)&=\stirlingfk{n}{k}\ \ (\textit{Stirling number of first kind}),\\
    B_{n,k}(1!,...,(n-k)!)&=\binom{n-1}{k-1}\frac{n!}{k!}\ \ (\textit{Lah number}),\\
    B_{n,k}(1,1,...,1)&=\stirlingsk{n}{k}\ \ (\textit{Stirling number of second kind}),\\
    B_{n,k}(1,2,...,n-k+1)&=\binom{n}{k}k^{n-k}\ \ (\textit{Idempotent number}).
\end{align*}
Then we can see the beautiful relationship that exists between Bell polynomials and numbers like the above.

On the other hand, the partial Bell polynomials can be efficiently computed by means of the recurrence relation
\begin{equation}\label{eqn_recu_bep}
    B_{n,k}(x_{1},...,x_{n-k+1})=\sum_{i=1}^{n-k+1}\binom{n-1}{i-1}x_{i}B_{n-i,k-1}(x_{1},...,x_{n-i-k+2}).
\end{equation}

The definition of complete Bell polynomials is as follows
\begin{definition}
The sum 
\begin{equation}
B_{n}(x_{1},x_{2},...,x_{n})=\sum_{k=1}^{n}B_{n,k}(x_{1},x_{2},...,x_{n-k+1})
\end{equation}
is called $n$-th complete exponential Bell polynomials with exponential generating function given by
to make $u=1$ in (\ref{eqn_egf})
\begin{equation}
\exp\left(\sum_{m=1}^{\infty}x_{m}\frac{t^{m}}{m!}\right)=\sum_{n=0}^{\infty}B_{n}(x_{1},x_{2},...,x_{n})\frac{t^{n}}{n!}
\end{equation}
and $B_{0}=1$.
\end{definition}
Some complete Bell polynomials are
\begin{align*}
B_{1}(x_{1})&=x_{1},\\
B_{2}(x_{1},x_{2})&=x_{1}^{2}+x_{2},\\
B_{3}(x_{1},x_{2},x_{3})&=x_{1}^{3}+3x_{1}x_{2}+x_{3},\\
B_{4}(x_{1},x_{2},x_{3},x_{4})&=x_{1}^{4}+6x_{1}^{2}x_{2}+4x_{1}x_{3}+3x_{2}^{2}+x_{4},\\
B_{5}(x_{1},x_{2},x_{3},x_{4},x_{5})&=x_{1}^{5}+10x_{1}^{3}x_{2}+15x_{1}x_{2}^{2}+10x_{1}^{2}x_{3}+10x_{2}x_{3}+5x_{1}x_{4}+x_{5}.
\end{align*}

\begin{theorem}
The complete Bell polynomials $B_{n}$ satisfy the identity
\begin{equation}\label{eqn_recu_bec}
    B_{n+1}(x_{1},...,x_{n+1})=\sum_{i=0}^{n}\binom{n}{i}B_{n-i}(x_{1},...,x_{n-i})x_{i+1}.
\end{equation}
\end{theorem}

From this it follows that
\begin{equation}\label{eqn_bell_0}
    B_{2n+1}(0,x_{2},0,..,0,x_{2n+1})=0
\end{equation}
for all $n\geq0$.

Another useful identity that Bell polynomials fulfill is as follows
\begin{equation}\label{eqn_bell_signo}
    B_{n}(-x_{1},x_{2},-x_{3},...,(-1)^{n-1}x_{n})=(-1)^{n}B_{n}(x_{1},x_{2},x_{3},...,x_{n})
\end{equation}

\section{Exponential autonomous functions}

We will study the solution of the equation 
\begin{equation}\label{eqn_exp_k}
y^{(k)}=e^{ay}    
\end{equation}
for any $a\in\mathbb{C}$. Making $y=u/a$ we obtain the equivalent equation 
\begin{equation}\label{eqn_equi}
u^{(k)}=ae^{u}   
\end{equation}
Then without loss of generality we will focus on the equation (\ref{eqn_equi}). Now by applying derivative to (\ref{eqn_equi}) we obtain another equation equivalent to (\ref{eqn_exp_k})
\begin{eqnarray}
    u^{(k+1)}&=&ae^{u}u^{\prime}=u^{(k)}u^{\prime}
\end{eqnarray}
Denote $(x_{1},x_{2},...,x_{k})$ the initial value problem $y(0)=x_{1}$, $y^{\prime}(0)=x_{2}$,...,$y^{(k-1)}(0)=x_{k}$. In this section the general solution and solutions with initial values $(x,0,0,...,0)$, $(x_{1}+k\ln c,cx_{2},...,c^{k-1}x_{k})$, and $(x_{1},-x_{2},...,x_{2k-1},-x_{2k})$ of the equation (\ref{eqn_equi}) are given. We will define the complete and partial exponential autonomous functions and the exponential autonomous polynomials, which are the coefficients of the power series solution of the equation (\ref{eqn_equi}). Moreover, we will find special values of these functions by using Bell polynomials. We begin with the following definition
\begin{definition}
Take $a\in\C$. Suppose $\x=(x_{1},...,x_{k})$. Let $f_{n}(\x,a)$ denote the \textit{complete exponential autonomous functions of order $k$}, $k\geq1$, recursively defined as
\begin{align}
    f_{0}(\x,a)&=x_{1},\\
    f_{1}(\x,a)&=x_{2},\\
    &\vdots&\nonumber\\
    f_{k-1}(\x,a)&=x_{k},\\
    f_{k}(\x,a)&=ae^{x_{1}},\\
    f_{n+k}(\x,a)&=ae^{x_{1}}B_{n}(f_{1}(\x,a),...,f_{n}(\x,a)),\ \ n\geq1, \label{eqn_def_fac}
\end{align}
where $B_{n}(y_{1},...,y_{n})$ are the complete Bell polynomials. When $x_{1}=0$, we define the \textit{exponential autonomous polynomials} as $q_{n}(x_{2},...,x_{k})=f_{n}(0,x_{2},...,x_{k})$, 
for $n\geq1$. 
\end{definition}
When $a=1$ in the above definition, we will write $f_{n}(\x)=f_{n}(\x,1)$. In this section we will restrict ourselves to exponential autonomous functions and autonomous polynomials will be dealt with in the next section.

The following are complete exponential autonomous functions for $k=1,2,3,4$. They will be very useful in the next section:\\
When $k=1$, $f_{n}(x,a)=(n-1)!a^{n}e^{nx}$.
When $k=2$
\begin{align*}
f_{0}(x,y,a)&=x,\\
f_{1}(x,y,a)&=y,\\
f_{2}(x,y,a)&=ae^{x},\\
f_{3}(x,y,a)&=aye^{x},\\
f_{4}(x,y,a)&=ae^{x}(ae^{x}+y^{2}),\\
f_{5}(x,y,a)&=ae^{x}(4aye^{x}+y^{3}),\\
f_{6}(x,y,a)&=ae^{x}(4a^{2}e^{2x}+11ay^{2}e^{x}+y^{4}).
\end{align*}
When $k=3$
\begin{align*}
f_{0}(x,y,z,a)&=x,\\
f_{1}(x,y,z,a)&=y,\\
f_{2}(x,y,z,a)&=z,\\
f_{3}(x,y,z,a)&=ae^{x},\\
f_{4}(x,y,z,a)&=aye^{x},\\
f_{5}(x,y,z,a)&=ae^{x}(z+y^{2}),\\
f_{6}(x,y,z,a)&=ae^{x}(ae^{x}+3yz+y^{3}),\\
f_{7}(x,y,z,a)&=ae^{x}(5ye^{x}+3z^{2}+6y^{2}z+y^{4}).
\end{align*}
And finally, when $k=4$
\begin{align*}
f_{0}(x,y,z,w,a)&=x,\\
f_{1}(x,y,z,w,a)&=y,\\
f_{2}(x,y,z,w,a)&=z,\\
f_{3}(x,y,z,w,a)&=w,\\
f_{4}(x,y,z,w,a)&=ae^{x},\\
f_{5}(x,y,z,w,a)&=aye^{x},\\
f_{6}(x,y,z,w,a)&=ae^{x}(z+y^{2}),\\
f_{7}(x,y,z,w,a)&=ae^{x}(w+3yz+y^{3}),\\
f_{8}(x,y,z,w,a)&=ae^{x}(ae^{x}+3z^{2}+4yw+6y^{2}z+y^{4}).
\end{align*}

In the following result we will show that the exponential generating function of the complete exponential autonomous functions is solution of the equation (\ref{eqn_equi})

\begin{theorem}\label{theo_sol}
Let $\textbf{x}=(x_{1},...,x_{k})$. The series 
\begin{equation}
    E_{k}(t,\textbf{x},a)=\sum_{n=0}^{\infty}f_{n}(\textbf{x},a)\frac{t^{n}}{n!}
\end{equation}
is solution of the differential equation (\ref{eqn_equi}).
\end{theorem}
\begin{proof}
Taking derivative $k$ times with respect to $t$ the series $E_{k}(t,\textbf{x},a)$, using the definition of the autonomous functions $f_{n}(\textbf{x},a)$ and the equation (\ref{eqn_egf}), then
\begin{align*}
    \frac{\partial^{k}E_{k}(t,\textbf{x},a)}{\partial t^{k}}&=\sum_{n=0}^{\infty}f_{n+k}(\textbf{x},a)\frac{t^{n}}{n!}\\
    &=ae^{x_{1}}+\sum_{n=1}^{\infty}ae^{x_{1}}B_{n}(f_{1}(\x,a),...,f_{n}(\x,a))\frac{t^{n}}{n!}\\
    &=e^{ax_{1}}\left(1+\sum_{n=1}^{\infty}B_{n}(f_{1}(\x,a),...,f_{n}(\x,a))\frac{t^{n}}{n!}\right)\\
    &=ae^{x_{1}}e^{E_{k}(t,\textbf{x},a)-x_{1}}\\
    &=ae^{E_{k}(t,\textbf{x},a)}.
\end{align*}
\end{proof}

Now we define the partial exponential autonomous functions
\begin{definition}
Let $g_{n,i}(\x,a)$ denote the \textit{partial exponential autonomous functions} as
\begin{equation}
    g_{n,i}(\x,a)=B_{n,i}(f_{1}(\x,a),...,f_{n-i+1}(\x,a))
\end{equation}
with $g_{0,0}(\x,a)=1$, $g_{n,0}(\x,a)=0$, for $n\geq1$, and $g_{0,i}(\x,a)=0$ , for $i\geq1$. Then
\begin{equation}
    f_{n+k}(\x,a)=ae^{x_{1}}\sum_{i=1}^{n}g_{n,i}(\x,a).
\end{equation}
\end{definition}

In the following result we establish recurrence relations for the functions $f_{n}(\x,a)$ and $g_{n,i}(\x,a)$. Many important results of this paper will be proved using this theorem.
\begin{theorem}\label{theo_recu}
The autonomous functions $f_{n}(\textbf{x},a)$ and $g_{n,i}(\x,a)$ fulfill the recurrence relations
\begin{equation}\label{eqn_recu_fac}
    f_{n+k+1}(\textbf{x},a)=\sum_{i=0}^{n}\binom{n}{i}f_{n-i+k}(\textbf{x},a)f_{i+1}(\textbf{x},a)
\end{equation}
and
\begin{equation}\label{eqn_recu_fap}
    g_{n,i}(\textbf{x},a)=\sum_{j=1}^{n-i+1}\binom{n-1}{j-1}f_{j}(\textbf{x},a)g_{n-j,i-1}(\textbf{x},a).
\end{equation}
\end{theorem}
\begin{proof}
Making $y_{j}=f_{j}(\textbf{x},a)$ in (\ref{eqn_recu_bep}) and (\ref{eqn_recu_bec}) and multiplying these by $ae^{x_{1}}$, we obtain the desired result.
\end{proof}
Now we will study the behavior of the functions $f_{n}(\x,a)$ evaluated at $\x=(x_{1},0,0,...,0)$. From previous result we can construct the first important sequence arising from the differential equation (\ref{eqn_equi})
\begin{theorem}\label{theo_val_f}
The functions $f_{n}(\x,a)$ take the following values at $\x=(x_{1},0,...,0)$
\begin{enumerate}
    \item $f_{kn+1}(x_{1},0,...,0,a)=f_{kn+2}(x_{1},0,...,0,a)=\cdots=f_{kn+k-1}(x_{1},0,...,0,a)=0$, $n\geq0$,
    \item $f_{kn}(x_{1},0,...,0,a)=A_{n}^{(k)}(a)e^{nx_{1}}$, $n\geq1$
\end{enumerate}
where $A^{(k)}_{1}(a)=1$ and
\begin{equation}\label{eqn_recu_A}
    A^{(k)}_{n+2}(a)=\sum_{i=0}^{n}\binom{kn+k-1}{ki+k-1}A^{(k)}_{n-i+1}(a)A^{(k)}_{i+1}(a),\ \ 
\end{equation}
$n\geq0$, $k\geq1$.
\end{theorem}
\begin{proof}
Let $\mathbf{x}=(x_{1},0,...,0)$. Clearly, $f_{1}(\mathbf{x},a)=0$, $f_{k+1}(\mathbf{x},a)=B_{1}(f_{1}(\mathbf{x},a))=f_{1}(\mathbf{x},a)=0$.  Now suppose it is true that $f_{ki+1}(\mathbf{x},a)=0$ for $2\leq i\leq n-1$. By the theorem \ref{theo_recu}
\begin{align*}
f_{kn+1}(\mathbf{x},a)&=f_{k(n-1)+k+1}(\mathbf{x},a)\\
&=\sum_{i=0}^{k(n-1)}\binom{k(n-1)}{i}f_{kn-i}(\mathbf{x},a)f_{i+1}(\mathbf{x},a).
\end{align*}
Since the product $f_{kn-i}(\mathbf{x},a)f_{i+1}(\mathbf{x},a)$ contain the functions $f_{kj+1}(\mathbf{x},a)$, then $f_{kn+1}(\mathbf{x},a)=0$ for all $n$. Likewise it is proved for $f_{kn+j}(\mathbf{x},a)$, $j=2,...,k-1$. Now we will prove $2$. We know that $f_{k}(\mathbf{x},a)=ae^{x_{1}}$, $f_{2k}(\mathbf{x},a)=a^{2}e^{2x_{1}}$ and suppose that $f_{kn}(\mathbf{x},a)=A_{n}^{(k)}(a)e^{nx_{1}}$. Then
\begin{align*}
    f_{kn+k}(\mathbf{x},a)&=f_{(kn-1)+k+1}(\mathbf{x},a)\\
    &=\sum_{i=0}^{kn-1}\binom{kn-1}{i}f_{kn-1-i+k}(\mathbf{x},a)f_{i+1}(\mathbf{x},a)\\
    &=\sum_{i=0}^{n-1}\binom{kn-1}{ki+1}f_{k(n-i)}(\mathbf{x},a)f_{ki+k}(\mathbf{x},a)\\
    &=\sum_{i=0}^{n-1}\binom{kn-1}{ki+1}A_{n-i}^{(k)}(a)e^{(n-i)x_{1}}A_{i+1}^{(k)}(a)e^{(i+1)x_{1}}\\
    &=e^{(n+1)x_{1}}\sum_{i=0}^{n-1}\binom{kn-1}{ki+1}A_{n-i}^{(k)}(a)A_{i+1}^{(k)}(a)\\
    &=e^{(n+1)x_{1}}A_{n+1}^{(k)}(a).
\end{align*}
\end{proof}

It is easy to show that $A_{n}^{(1)}(a)=(n-1)!a^{n}$ when $k=1$. We will use the equation (\ref{eqn_recu_A}) to prove this result. Suppose it is true that 
$A_{i}^{(1)}(a)=(i-1)!a^{i}$ for $i$ ranging between $1$ and $n+1$. We have
\begin{align*}
A^{(1)}_{n+2}(a)&=\sum_{i=0}^{n}\binom{n}{i}A^{(1)}_{n-i+1}(a)A^{(1)}_{i+1}(a)\\
&=\sum_{i=0}^{n}\binom{n}{i}(n-i)!a^{n-i+1}i!a^{i+1}\\
&=a^{n+2}\sum_{i=0}^{n}\binom{n}{i}(n-i)!i!\\
&=a^{n+2}n!(n+1)=a^{n+2}(n+1)!.
\end{align*}
We can extend the above result to all $k\geq1$
\begin{proposition}
For all $k\geq1$ we have
\begin{equation}
    A_{n}^{(k)}(a)=a^{n}A_{n}^{(k)}(1)
\end{equation}
\end{proposition}
\begin{proof}
Suppose by induction that $A_{i}^{(k)}(a)=a^{i}A_{i}^{(k)}(1)$ for all $i$ ranging between $1$ and $n+1$, then use the same steps as in the previous proof.
\end{proof}

From the above proposition it follows that 
\begin{equation}
    E_{k}(t,(x,0,...,0),a)=E_{k}(at,(x,0,...,0),1).
\end{equation}
Then without loss of generality it is sufficient to study the solution of (\ref{eqn_equi}) with initial conditions $y(0)=x$, $y^{\prime}(0)=y^{\prime}(0)=\cdots=y^{(k-1)}(0)=0$ and $a=1$ to generate the sequence $A_{n}^{(k)}(1)$.

The following corollary of theorem \ref{theo_val_f} shows us that the numbers $A_{n}^{(k)}$ can be constructed using Bell polynomials
\begin{corollary}\label{cor_val_f}
Numbers $A_{n}^{(k)}(a)$ fulfill the recurrence relation
\begin{equation}\label{eqn_bt}
    A_{n}^{(k)}(a)=B_{n-k}(\overbrace{0,...,0}^{k-1},A_{1}^{(k)}(a),...,\overbrace{0,...,0}^{k-1},A_{n-k}^{(k)}(a)),\ \ n\geq1.
\end{equation}
\end{corollary}

In the following theorem we calculate some special values of the functions
$g_{n.i}(\x,a)$
\begin{theorem}
Let $\x=(x_{1},0,...,0)$. For all $a\in\R$ we have
\begin{enumerate}
    \item $g_{n,i}(\x,a)=0$, si $k\not\vert n$.
    \item $g_{lk,1}(\x,a)=A_{l}^{(k)}(a)e^{lx_{1}}$.
    \item $g_{lk,2}(\x,a)=e^{lx_{1}}\sum_{j=1}^{l}\binom{kl-1}{kj-1}A_{j}^{(k)}(a)A_{l-j}^{(k)}(a)$.
    \item $g_{n,n}(\x,a)=g_{n,n-1}(\x,a)=g_{n,n-2}(\x,a)=g_{n,n-3}(\x,a)=g_{n,n-4}(\x,a)=0$, $k>1$.
\end{enumerate}
\end{theorem}
\begin{proof}
Suppose $k\not\vert n$ and $g_{n-j,i-1}(\x,a)=0$ for all $k$ such that $k\not\vert j$. Using theorem \ref{theo_recu} and \ref{theo_val_f} we have that $f_{j}(\x,a)=0$. This proves 1. To prove 2 we have
\begin{align*}
    g_{lk,1}(\x,a)&=\sum_{j=1}^{lk}\binom{lk-1}{j-1}f_{j}(\x,a)g_{lk-j,0}(\x,a)\\
    &=\binom{lk-1}{lk-1}f_{lk}(\x,a)g_{0,0}(\x,a)\\
    &=A_{l}^{(k)}(a)e^{lx_{1}}.
\end{align*}
On the other hand, 
\begin{align*}
    g_{lk,2}(\x,a)&=\sum_{j=1}^{lk-1}\binom{lk-1}{j-1}f_{j}(\x,a)g_{lk-j,1}(\x,a)\\
    &=\sum_{j=1}^{l}\binom{lk-1}{jk-1}f_{kj}(\x,a)g_{lk-kj,1}(\x,a)\\
    &=\sum_{j=1}^{l}\binom{lk-1}{jk-1}A_{j}^{(k)}(a)e^{jx_{1}}A_{l-j}^{(k)}(a)e^{(l-j)x_{1}}\\
    &=e^{lx_{1}}\sum_{j=1}^{l}\binom{lk-1}{jk-1}A_{j}^{(k)}(a)A_{l-j}^{(k)}(a).
\end{align*}
Then this proves 3. To prove 4 we use the equations (\ref{eqn_n_n})-(\ref{eqn_n,n-4}).
\end{proof}

We conclude this section with the following properties of the exponential autonomous functions
\begin{theorem}\label{theo_contra}
For all $n\geq1$, $k\geq1$ and for all $a,c\in\C$ is fulfilled
\begin{equation}
    f_{n}(x_{1}+k\ln c,cx_{2},...,c^{k-1}x_{k},a)=c^{n}f_{n}(x_{1},x_{2},...,x_{k},a)
\end{equation}
\end{theorem}
\begin{proof}
Let $\y=(x_{1}+k\ln c,cx_{2},...,c^{k-1}x_{k})$ and $\x=(x_{1},x_{2},...,x_{k})$. Suppose that the result is true for $i\leq n$. Then
\begin{align*}
f_{n+1}(\y,a)&=f_{(n+1-k)+k}(\y,a)\\
&=ac^{k}e^{x_{1}}B_{n+1-k}(f_{1}(\y,a),...,f_{n+1-k}(\y,a))\\
&=ac^{k}e^{x_{1}}B_{n+1-k}(cf_{1}(\x,a),...,c^{n+1-k}f_{n+1-k}(\x,a))\\
&=ac^{k}c^{n+1-k}e^{x_{1}}B_{n}(f_{1}(\x,a),...,f_{n}(\x,a))\\
&=c^{n+1}f_{n+1}(\x,a).
\end{align*}
\end{proof}

The following is the corollary to the theorem \ref{theo_contra} that allows to calculate the solutions of (\ref{eqn_equi}) when the initial values are $(x_{1}+k\ln c,cx_{2},...,c^{k-1}x_{k})$
\begin{corollary}
\begin{equation}
    E_{k}(t,(x_{1}+k\ln c,cx_{2},...,c^{k-1}x_{k}),a)=k\ln c+E_{k}(ct,(x_{1},x_{2},...,x_{k}),a)
\end{equation}
\end{corollary}
\begin{proof}
Let $\y=(x_{1}+k\ln c,cx_{2},...,c^{k-1}x_{k})$ and $\x=(x_{1},x_{2},...,x_{k})$. From the above theorem and the definition of the function $E_{k}(t,\x,a)$ we have
\begin{align*}
E_{k}(t,\y,a)&=x_{1}+k\ln c+\sum_{n=1}^{\infty}f_{n}(\y,a)\frac{t^{n}}{n!}\\
&=x_{1}+k\ln c+\sum_{n=1}^{\infty}c^{n}f_{n}(\x,a)\frac{t^{n}}{n!}\\
&=x_{1}+k\ln c+\sum_{n=1}^{\infty}f_{n}(\x,a)\frac{(ct)^{n}}{n!}\\
&=k\ln c+E_{k}(ct,\x,a).
\end{align*}
\end{proof}

Finally, we compute $f_{n}(\x,a)$ when $\x=(x_{1},-x_{2},...,x_{2k-1},-x_{2k})$
\begin{theorem}\label{theo_alt}
For all $n\geq0$ and for all exponential autonomous functions of order $2k$ it is satisfied that
\begin{equation}
    f_{n}((x_{1},-x_{2},...,x_{2k-1},-x_{2k}),a)=(-1)^{n}f_{n}((x_{1},x_{2},...,x_{2k-1},x_{2k}),a)
\end{equation}
\end{theorem}
\begin{proof}
Let $\y=(x_{1},-x_{2},...,x_{2k-1},-x_{2k})$, and let $\x=(x_{1},x_{2},...,x_{2k-1},x_{2k})$. Suppose it is true for all values less than or equal to $n$. Then
\begin{eqnarray*}
f_{n+1}(\y,a)&=&f_{(n+1-2k)+2k}(\y,a)\\
&=&ae^{x_{1}}B_{n+1-2k}(f_{1}(\y,a),f_{2}(\y,a),...,f_{n+1-2k}(\y,a))\\
&=&ae^{x_{1}}B_{n+1-2k}(-f_{1}(\x,a),f_{2}(\x,a),...,(-1)^{n+1-2k}f_{n+1-2k}(\x,a))\\
&=&(-1)^{n+1-2k}ae^{x_{1}}B_{n+1-2k}(f_{1}(\x,a),f_{2}(\x,a),...,f_{n+1-2k}(\x,a))\\
&=&(-1)^{n+1}f_{n}(\x,a).
\end{eqnarray*}
\end{proof}

Finally we have the corollary to theorem \ref{theo_alt}

\begin{corollary}\label{cor_sig_alt}
\begin{equation}
E_{2k}(-t,(x_{1},-x_{2},...,x_{2k-1},-x_{2k}),a)=E_{2k}(t,(x_{1},x_{2},...,x_{2k-1},x_{2k}),a).
\end{equation}
\end{corollary}
\begin{proof}
From the above theorem and the definition of the function $E_{2k}(t,x,a)$ we have
\begin{align*}
E_{2k}(-t,\y,a)&=\sum_{n=0}^{\infty}f_{n}(\y,a)\frac{(-t)^{n}}{n!}\\
&=\sum_{n=0}^{\infty}(-1)^{n}(-1)^{n}f_{n}(\x,a)\frac{t^{n}}{n!}\\
&=\sum_{n=0}^{\infty}f_{n}(\x,a)\frac{t^{n}}{n!}\\
&=E_{2k}(t,\x,a)
\end{align*}
where $\y=(x_{1},-x_{2},...,x_{2k-1},-x_{2k})$ and $\x=(x_{1},x_{2},...,x_{2k-1},x_{2k})$.
\end{proof}

\section{(k,a)-autonomous coefficients}

When $k=1$ we obtain the equation $y^{\prime}=ae^{y}$, which is the easiest to solve for all $k$. Using the method of separation of variables we reach the solution
\begin{equation*}
    y(t)=-\ln(e^{-x}-at)
\end{equation*}
with initial condition $y(0)=x$. On the other hand, by the theorem \ref{theo_sol} the solution in power series becomes
\begin{align*}
E_{1}(t,x,a)&=x+\sum_{n=1}^{\infty}A_{n}^{(1)}(a)\frac{t^{n}}{n!}\\
&=x+\sum_{n=1}^{\infty}(n-1)!a^{n}e^{nx}\frac{t^{n}}{n!}\\
&=x+\sum_{n=1}^{\infty}\frac{(ae^{x})^{n}}{n}=x-\ln(1-ae^{x}t)
\end{align*}
Now we can use the results of the previous section to prove some results already known. By the definition of complete exponential autonomous functions
\begin{align*}
n!a^{n+1}e^{a(n+1)x}&=ae^{x}\sum_{i=1}^{n}B_{n,i}(0!a^{1}e^{x},1!a^{2}e^{2x},...,(n-i)!a^{n-i+1}e^{(n-i+1)x})\\
&=ae^{x}a^{n}e^{nx}\sum_{i=1}^{n}B_{n,i}(0!,1!,...,(n-i)!)\\
&=e^{a(n+1)x}a^{n+1}\sum_{i=1}^{n}B_{n,i}(0!,1!,...,(n-i)!)
=e^{a(n+1)x}a^{n+1}\sum_{i=1}^{n}\stirlingfk{n}{i}
\end{align*}
from which follows the result relating factorials and Stirling number of first kind
\begin{equation}\label{eqn_fac_stir}
    n!=\sum_{i=1}^{n}\stirlingfk{n}{i}.
\end{equation}
Furthermore $g_{n,i}(\x,1)=\stirlingfk{n}{i}$ and by the equation (\ref{eqn_recu_fap}) we obtain the following finite-sum identity
\begin{align*}
    \stirlingfk{n+1}{i+1}&=\sum_{j=1}^{n-i+1}\binom{n}{j-1}(j-1)!\stirlingfk{n+1-j}{i}\\
    &=\sum_{j=n}^{i}\frac{n!}{j!}\stirlingfk{j}{i}\\
    &=\sum_{j=i}^{n}\frac{n!}{j!}\stirlingfk{j}{i}\\
    &=\sum_{j=0}^{n}\frac{n!}{j!}\stirlingfk{j}{i}.
\end{align*}
On the other hand, from the equation (\ref{eqn_recu_fac}) we obtain the trivial result
\begin{align*}
(n+1)!&=\sum_{i=0}^{n}\binom{n}{i}(n-i)!i!\\
&=\sum_{i=0}^{n}n!.
\end{align*}
The Stirling numbers of the first kind originally arose algebraically from the expansion of the falling factorial
\begin{equation*}
    (x)_{n}=x(x-1)(x-2)\cdots(x-n+1)
\end{equation*}
and in polynomial form is as follows
\begin{equation*}
    (x)_{n}=\sum_{i=0}^{n}(-1)^{n-k}\stirlingfk{n}{i}x^{i}.
\end{equation*}
Analogously, we want to define and study the coefficients of the expansion of the autonomous exponential polynomials $q_{n}(\x,a)$ with $\x=(0,x,x,...,x)$. First we calculate the degree of $q_{n}(\x,a)$.

\begin{proposition}
Let $\x=(0,x,x,...,x)$. Then the degree $q_{n}(\x,a)$ of
\begin{equation}
\gr(q_{n}(\x,a))=n-k,\ \ \ n\geq k.
\end{equation}
\end{proposition}
\begin{proof}
By definition
\begin{align*}
q_{n+k}(\x,a)&=\sum_{i=1}^{n-1}a^{i}g_{n,i}(\x,a)+a^{n}g_{n,n}(\x,a)\\
&=\sum_{i=1}^{n-1}a^{i}g_{n,i}(\x,a)+a^{n}x_{1}^{n}.
\end{align*}
As $\gr(g_{n,i}(\x,a))\leq i$, then $\gr(q_{n+k}(\x,a))=n$.
\end{proof}

We now define the autonomous polynomials and autonomous coefficients

\begin{definition}
Let $A_{n}^{(k)}(x,a)=q_{n}(0,x,...,x,a)$ denote the \textit{autonomous polynomials} of degree $n-k$ for all $n\geq k$.
\end{definition}

Using (\ref{eqn_def_fac}) we note that
\begin{equation}\label{eqn_pol_auto}
    A_{n+k}^{(k)}(x,a)=aB_{n}(A_{1}^{(k)}(x,a),...,A_{n}^{(k)}(x,a))
\end{equation}
for all $n\geq1$. 

\begin{definition}
We define the \textit{$(k,a)$-autonomous coefficients}, denoted by $\cfauto{n}{i}_{(k,a)}$, as the coefficients of the autonomous polynomials $A_{n+k}^{(k)}(x,a)$, i.e.,
\begin{equation}
    A_{n+k}^{(k)}(x,a)=\sum_{i=0}^{n}\cfauto{n}{i}_{(k,a)}x^{i}.
\end{equation}
\end{definition}

Now we will give some values of the $(k,a)$-autonomous coefficients
\begin{theorem}\label{theo_val}
Some values of the coefficients $\cfauto{n}{i}_{(k,a)}$ are
\begin{align}
\cfauto{n}{0}_{(k,a)}&=
\begin{cases}
0,& \text{ si $k\not\vert n$;}\\
a^{n/k}A_{n/k}^{(k)}(1),& \text{ si $k\vert n$} 
\end{cases}\label{eqn_sti_n_0}\\
\cfauto{0}{i}_{(k,a)}&=0, \text{ si $i\geq1$},\label{eqn_sti_0_i}\\
\cfauto{n}{n-l}_{(k,a)}&=a\stirlingsk{n}{n-l},\ k> l+1,\ 0\leq l<n,\label{eqn_sti_n_n-l}
\end{align}
\end{theorem}
\begin{proof}
The equation (\ref{eqn_sti_n_0}) follows from theorem \ref{theo_val_f}.  By definition, $B_{0,i}=0$ for $i\geq1$. Then (\ref{eqn_sti_0_i}) is true. Finally, if $k>l+1$,
\begin{align*}
    B_{n,n-l}(A_{1}^{(k)}(x,a),...,A_{l+1}^{(k)}(x,a))&=B_{n,n-l}(x,...,x)\\
    &=\stirlingsk{n}{n-l}x^{n-l}
\end{align*}
and from here follows (\ref{eqn_sti_n_n-l}).
\end{proof}
We will now show the relationship between the $(k,1)$-autonomous coefficients and the binomial coefficients
\begin{theorem}
\begin{align}
    \cfauto{n+1}{0}_{(k,1)}&=\binom{n}{k-1}\cfauto{n+1-k}{0}_{(k,1)}\nonumber\\
    &+\sum_{h=k+1}^{n}\binom{n}{h}\cfauto{n-h}{0}_{(k,1)}\cfauto{h+1-k}{0}_{(k,1)},
\end{align}
for $1\leq i\leq n-k+3$
\begin{align}
    \cfauto{n+1}{i}_{(k,1)}&=\sum_{h=0}^{k-2}\binom{n}{h}\cfauto{n-h}{i-1}_{(k,1)}\nonumber\\
    &+\binom{n}{k-1}\cfauto{n+1-k}{i}_{(k,1)}+\binom{n}{k}\cfauto{n-k}{i-1}_{(k,1)}\nonumber\\
    &+\sum_{h=k+1}^{n}\binom{n}{h}\sum_{j+l=i}\cfauto{n-h}{j}_{(k,1)}\cfauto{h+1-k}{l}_{(k,1)}
\end{align}
and for $n-k+4\leq i\leq n+1$
\begin{align}
    \cfauto{n+1}{i}_{(k,1)} &=\sum_{h=0}^{n-i+1}\binom{n}{h}\cfauto{n-h}{i-1}_{(k,1)}\nonumber\\
    &+\binom{n}{k-1}\cfauto{n+1-k}{i}_{(k,1)}+\binom{n}{k}\cfauto{n-k}{i-1}_{(k,1)}\nonumber\\
    &+\sum_{h=k+1}^{n}\binom{n}{h}\sum_{j+l=i}\cfauto{n-h}{j}_{(k,1)}\cfauto{h+1-k}{l}_{(k,1)}
\end{align}
\end{theorem}
\begin{proof}
As 
\begin{align*}
    A_{n+1+k}^{(k)}(x,1)&=\sum_{i=0}^{n}\binom{n}{i} A_{n-i+k}^{(k)}(x,1)A_{i+1}^{(k)}(x,1)\\
    &=\sum_{i=0}^{k-2}\binom{n}{i}A_{n-i+k}^{(k)}(x,1)x+\binom{n}{k-1}A_{n+1}^{(k)}(x,1)\\
    &+\binom{n}{k}A_{n}^{(k)}(x,1)x+\sum_{i=k+1}^{n}\binom{n}{i} A_{n-i+k}^{(k)}(x,1)A_{i+1}^{(k)}(x,1),
\end{align*}
then
\begin{align*}
    \sum_{i=0}^{n+1}\cfauto{n+1}{i}_{(k,1)}x^{i}&=\sum_{i=0}^{k-2}\binom{n}{i}\sum_{j=0}^{n-i}\cfauto{n-i}{j}_{(k,1)}x^{j+1}\\
    &+\binom{n}{k-1}\sum_{j=0}^{n+1-k}\cfauto{n+1-k}{j}_{(k,1)}x^{j}
    +\binom{n}{k}\sum_{j=0}^{n-k}\cfauto{n-k}{j}_{(k,1)}x^{j+1}\\
    &+\sum_{i=k+1}^{n}\binom{n}{i}\left(\sum_{j=0}^{n-i}\cfauto{n-i}{j}_{(k,1)}x^{j}\sum_{l=0}^{i+1-k}\cfauto{i+1-k}{l}_{(k,1)}x^{l}\right).
\end{align*}
We multiply the two autonomous polynomials within the last sum
\begin{align*}
\sum_{i=0}^{n+1}\cfauto{n+1}{i}_{(k,1)}x^{i}
    &=\sum_{i=0}^{k-2}\binom{n}{i}\sum_{j=0}^{n-i}\cfauto{n-i}{j}_{(k,1)}x^{j+1}\\
    &+\binom{n}{k-1}\sum_{j=0}^{n+1-k}\cfauto{n+1-k}{j}_{(k,1)}x^{j}
    +\binom{n}{k}\sum_{j=0}^{n-k}\cfauto{n-k}{j}_{(k,1)}x^{j+1}\\
    &+\sum_{i=k+1}^{n}\binom{n}{i}\sum_{h=0}^{n+1-k}\left(\sum_{j+l=h}\cfauto{n-i}{j}_{(k,1)}\cfauto{i+1-k}{l}_{(k,1)}\right)x^{h}.
\end{align*}
Then by rearranging the first and fourth sums we obtain
\begin{align*}
    \sum_{i=0}^{n+1}\cfauto{n+1}{i}_{(k,1)}x^{i}
    &=\sum_{i=0}^{n-k+3}\left(\sum_{h=0}^{k-2}\binom{n}{h}\cfauto{n-h}{i-1}_{(k,1)}\right)x^{i}\\
    &+\sum_{i=n-k+4}^{n+1}\left(\sum_{h=0}^{n+1-i}\binom{n}{h}\cfauto{n-h}{i-1}_{(k,1)}\right)x^{i}\\
    &+\binom{n}{k-1}\sum_{i=0}^{n+1-k}\cfauto{n+1-k}{i}_{(k,1)}x^{i}\\
    &+\binom{n}{k}\sum_{i=1}^{n-k+1}\cfauto{n-k}{i-1}_{(k,1)}x^{i}\\
    &+\sum_{i=0}^{n+1-k}\left(\sum_{h=k+1}^{n}\binom{n}{h}\sum_{j+l=i}\cfauto{n-h}{j}_{(k,1)}\cfauto{h+1-k}{l}_{(k,1)}\right)x^{i}.
\end{align*}
For a suitable value of $i$ the desired results are attained.
\end{proof}

Finally, we show without proof the relationship between the $(k,1)$-autonomous coefficients and the Stirling numbers of second kind. 
\begin{conjecture}
Suppose that $A_{1}^{(k)}(1,1)=\cdots=A_{k}^{(k)}(1,1)=1$. Then
\begin{equation}
    B_{n}(A_{1}^{(k)}(1,1),...,A_{n}^{(k)}(1,1))=\sum_{i=1}^{n}\stirlingsk{n}{i}A_{i}^{(k)}(1,1),\ \ n\geq1
\end{equation}
Then,
\begin{equation}\label{eqn_conj1_1}
    A_{n+k}^{(k)}(1,1)=\sum_{i=1}^{n}\stirlingsk{n}{i}A_{i}^{(k)}(1,1),\ \ n\geq1
\end{equation}
and
\begin{equation}\label{eqn_conj1_2}
    \sum_{i=0}^{n}\cfauto{n}{i}_{(k,1)}=\sum_{j=1}^{k}\stirlingsk{n}{j}+\sum_{j=k+1}^{n}\stirlingsk{n}{j}\sum_{i=0}^{j-k}\cfauto{n}{i}_{(k,1)}
\end{equation}
\end{conjecture}
Equation (\ref{eqn_conj1_1}) corresponds to the number of shifts left $k-1$ places under Stirling transform.

\section{Sequences related to the equation (\ref{eqn_equi})}

We conclude this article by showing sequences related to equation (\ref{eqn_equi}) for values of $k=2,3,4$. Especially, we show that the numbers known as reduced tangent numbers, Bernoulli numbers, Euler zigzag numbers, Eulerian numbers, Blasius numbers, triangular numbers, number of shifts left 3 places under Stirling transform, and number of 8-sequences of $[1,n]$ with 2 contiguous pairs can be constructed using Bell polynomials, Stirling numbers of second kind, binomial coefficient and autonomous coefficients. 

\subsection{Case k=2}

The first case to be studied is
\begin{equation}\label{eqn_k=2}
y^{\prime\prime}=ae^{y}
\end{equation}

The equation (\ref{eqn_k=2}) is equivalent to the equation
$y^{(3)}=y^{\prime\prime}y^{\prime}$, whose solution is
\begin{equation*}
    y^{\prime}=\sqrt{2}\sqrt{c_{1}}\tan\left(\frac{1}{2}\sqrt{2}\sqrt{c_{1}}t+\frac{1}{2}\sqrt{2}\sqrt{c_{1}}c_{2}\right)
\end{equation*}
and therefore
\begin{equation}\label{eqn_sol_cons}
    y=\ln\left(\sec^{2}\left(\frac{1}{2}\sqrt{2}\sqrt{c_{1}}t+\frac{1}{2}\sqrt{2}\sqrt{c_{1}}c_{2}\right)\right)+c_{3}
\end{equation}
where $c_{1},c_{2}$ and $c_{3}$ are constants in $\C$. Since we want $y(0)=x$, $y^{\prime}(0)=y$ and $y^{\prime}(0)=ae^{x}$, then
\begin{align*}
\ln\left(\sec^{2}\left(\frac{1}{2}\sqrt{2}\sqrt{c_{1}}c_{2}\right)\right)+c_{3}&=x,\\
\sqrt{2}\sqrt{c_{1}}\tan\left(\frac{1}{2}\sqrt{2}\sqrt{c_{1}}c_{2}\right)&=y,\\
c_{1}\sec^{2}\left(\frac{1}{2}\sqrt{2}\sqrt{c_{1}}c_{2}\right)&=ae^{x}.
\end{align*}
As result
\begin{align*}
c_{1}&=ae^{x}-\frac{y^{2}}{2},\\
\frac{1}{2}\sqrt{2}\sqrt{c_{1}}c_{2}&=\arctan\left(\frac{y}{\sqrt{2ae^{x}-y^{2}}}\right),\\
c_{3}&=x-\ln\left(1+\frac{y^{2}}{2ae^{x}-y^{2}}\right).
\end{align*}
Thus, the function
\begin{align}\label{eqn_sol_k=2}
    E_{2}(t,(x,y),a)&=x+\ln\sec^{2}\left(\frac{\sqrt{2ae^{x}-y^{2}}t}{2}+\arctan\left(\frac{y}{\sqrt{2ae^{x}-y^{2}}}\right)\right)\nonumber\\
    &-\ln\left(1+\frac{y^{2}}{2ae^{x}-y^{2}}\right)
\end{align}
is the solution of the equation (\ref{eqn_k=2}) with initial value $y(0)=x$, $y^{\prime}(0)=y$.

The following is a list of particular solutions of (\ref{eqn_k=2}) which are obtained from the equation (\ref{eqn_sol_k=2})
\begin{align}
E_{2}(t,(x,0),a)&=x+\ln\left(\sec^{2}\left(\frac{\sqrt{a}e^{x/2}t}{\sqrt{2}}\right)\right),\ \ a>0,\label{eqn_sol_x_0_p}\\
E_{2}(t,(x,0),-a)&=x+\ln\left(\sech^{2}\left(\frac{\sqrt{a}e^{x/2}t}{\sqrt{2}}\right)\right),\ \ a>0,\label{eqn_sol_x_0_m}\\
E_{2}(t,(0,y),a)&=\ln\sec^{2}\left(\frac{\sqrt{2a-y^{2}}t}{2}+\arctan\left(\frac{y}{\sqrt{2a-y^{2}}}\right)\right)\nonumber\\
    &-\ln\left(1+\frac{y^{2}}{2a-y^{2}}\right),\ \ a>0,\label{eqn_sol_0_y_p}\\
E_{2}(t,(0,y),-a)&=\ln\sech^{2}\left(\frac{\sqrt{2a+y^{2}}t}{2}+\arctanh\left(\frac{y}{\sqrt{2a+y^{2}}}\right)\right)\nonumber\\
    &-\ln\left(1-\frac{y^{2}}{2a+y^{2}}\right),\ \ a>0.\label{eqn_sol_0_y_m}
\end{align}

We now show the relationship between reduced tangent numbers (\seqnum{A002105} in OEIS) and Bell polynomials and binomial coefficients
\begin{theorem}\label{theo_tan_bell}
Let 
\begin{equation}
(T_{n})_{n\geq1}=(1,1,4,34,496,\ldots)    
\end{equation}
denote the sequence of reduced tangent numbers. Then
\begin{enumerate}
\item $A_{n}^{(2)}(a)=a^{n}T_{n}$.
\item $T_{n}=B_{n}(0,T_{1},\ldots,0,T_{n-1})$, $n\geq2$.
\item $(-1)^{n}T_{n}=B_{n}(0,-T_{1},\ldots,0,(-1)^{n-1}T_{n-1})$, $n\geq2$.
\item $T_{n+2}=\sum_{i=0}^{n}\binom{2n+1}{2i+1}T_{n-i+2}T_{i+1}$, $n\geq0$.
\end{enumerate}
\end{theorem}
\begin{proof}
Another way to write (\ref{eqn_sol_x_0_p}) is
\begin{equation*}
    E_{2}(t,(x,0),a)=x+\sqrt{2}\int_{0}^{\sqrt{a}e^{x/2}t}\tan\left(\frac{u}{\sqrt{2}}\right)du.
\end{equation*}
Then
\begin{align*}
    E_{2}(t,(x,0),a)&=x+\int_{0}^{\sqrt{a}e^{x/2}t}\sum_{n=1}^{\infty}T_{n}\frac{u^{2n-1}}{(2n-1)!}\\
    &=x+\sum_{n=1}^{\infty}T_{n}\frac{(\sqrt{a}e^{x/2}t)^{2n}}{(2n)!}.
\end{align*}
By comparison, $f_{2n}(x,0,a)=a^{n}T_{n}e^{nx}$. Thus follows 1, 2, and 3. The recurrence relation 4 follows from equation (\ref{eqn_recu_A}).
\end{proof}

In general, the solution (\ref{eqn_sol_x_0_p}) is the generating function of the sequence
\begin{equation*}
    (a^{n}T_{n})_{n\geq1}=(a,a^{2},4a^{3},34a^{4},496a^{5},\ldots).
\end{equation*}

On the other hand, it is known that $T_{n}=\frac{2^{n}(2^{2n}-1)\vert b_{2n}\vert}{n}$, where the
$b_{2n}$ are the Bernoulli numbers (\seqnum{A000367}, \seqnum{A002445} in OEIS). Then the theorem \ref{theo_tan_bell} provides a relation between Bell polynomials and Bernoulli numbers, that is

\begin{equation}
    \frac{2^{n}(2^{2n}-1)\vert b_{2n}\vert}{n}=B_{n}\left(0,6\vert b_{2}\vert,0,30\vert b_{4}\vert,\ldots,0,\frac{2^{n-1}(2^{2n-2}-1)\vert b_{2n-2}\vert}{n-1}\right)
\end{equation}

We now show the relationship between Euler zigzag numbers (\seqnum{A000111} in OEIS) and Bell polynomials, binomial coefficients, and Stirling numbers of second kind

\begin{theorem}
Suppose
\begin{equation}
(A_{n})_{n\geq0}=(1,1,1,2,5,16,61,272,\ldots)
\end{equation}
the sequence of Euler zigzag numbers. Then
\begin{enumerate}
\item $A_{n+1}=B_{n}(A_{0},\ldots,A_{n-1})$, $n\geq1$.
\item $(-1)^{n}A_{n+1}=B_{n}(-A_{0},A_{1},\ldots,(-1)^{n-1}A_{n-1})$, $n\geq1$.
\item $A_{n+2}=\sum_{i=0}^{n}\binom{n}{i}A_{n-i+1}A_{i}$, $\neq0$.
\item $A_{n+2}=\sum_{i=1}^{n}\stirlingsk{n}{i}A_{i}$, $n\geq1$.
\item $A_{2n+2}=T_{n}+\sum_{i=1}^{n}\cfauto{2n}{2i}_{(2,1)}$.
\item $A_{2n+3}=\sum_{i=1}^{n}\cfauto{2n+1}{2i+1}_{(2,1)}$
\end{enumerate}
\end{theorem}
\begin{proof}
From equation (\ref{eqn_sol_0_y_p}),
\begin{align*}
    E_{2}(t,(0,1),1)&=\ln\left(\sec^{2}\left(\frac{t}{2}+\frac{\pi}{4}\right)\right)-\ln(2)\\
    &=\ln(\sec^{2}(t)+\sec(t)\tan(t)).
\end{align*}
Then
\begin{align*}
E_{2}(t,(0,1),1)&=\int_{0}^{t}(\sec(u)+\tan(u))du\\
&=\int_{0}^{t}\left(1+\sum_{n=1}^{\infty}A_{n}\frac{u^{n}}{n!}\right)du\\
&=t+\sum_{n=1}^{\infty}A_{n}\int_{0}^{t}\frac{u^{n}}{n!}du\\
&=t+\sum_{n=1}^{\infty}A_{n}\frac{t^{n+1}}{(n+1)!}\\
&=\sum_{n=1}^{\infty}A_{n-1}\frac{t^{n}}{n!}.
\end{align*}
We apply the equation (\ref{eqn_def_fac}) to obtain 1. By corollary \ref{cor_sig_alt} it follows that
\begin{equation}
    E_{2}(t,(0,-1),1)=E(-t,(0,1),1)=\ln(\sec^{2}(t)-\sec(t)\tan(t)).
\end{equation}
From equation (\ref{eqn_bell_signo}) follows 2. Formula 3 follows from equation (\ref{eqn_recu_fac}). From equation (\ref{eqn_conj1_1}) follows 4. The identities 5 and 6 follow because the Euler zigzag numbers are obtained when $x=1$ in $A_{n}^{(2)}(x,1)$. 
\end{proof}

When $k=2$ the exponential autonomous polynomials and the autonomous polynomials match. Some autonomous polynomials of the equation (\ref{eqn_k=2}) are
\begin{eqnarray*}
q_{1}(y,a)&=&y,\\
q_{2}(y,a)&=&a,\\
q_{3}(y,a)&=&ay,\\
q_{4}(y,a)&=&a(a+y^{2}),\\
q_{5}(y,a)&=&a(4ay+y^{3}),\\
q_{6}(y,a)&=&a(4a^{2}+11ay^{2}+y^{4}),\\
q_{7}(y,a)&=&a(34a^{2}y+26ay^{3}+y^{5}),\\
q_{8}(y,a)&=&a(34a^{3}+180a^{2}y^{2}+57ay^{4}+y^{6}).
\end{eqnarray*}
From the above we obtain the first $(2,a)$-autonomous coefficients
\begin{table}[H]
    \centering
    \begin{tabular}{l|c|c|c|c|c|c|c}
         \backslashbox{$n$}{$i$}& 0&1 & 2& 3& 4& 5&6  \\\hline
         0& $a$&  &  &  & & &\\\hline
         1& 0&  $a$&  &  &  & &\\\hline
         2& $a^2$&  0&  $a$&  &  & &\\\hline
         3& 0&  $4a^2$&  0&  $a$&  & &\\\hline
         4& $4a^3$&  0&  $11a^2$&  0& $a$ & &\\
         5& 0&  $34a^{3}$&  0&  $26a^{2}$& 0 &$a$ &\\
         6& $34a^{4}$&  0&  $180a^{3}$&  0& $57a^{2}$ &0 & $a$
    \end{tabular}
    \caption{$(2,a)$-autonomous coefficients}
    \label{tab:my_label}
\end{table}

\begin{theorem}
Some values of $(2,a)$-autonomous coefficients are
\begin{eqnarray}
\cfauto{2n}{2i+1}_{(2,a)}&=&\cfauto{2n+1}{2i}_{(2,a)}=0\label{eqn_par_impar}
\end{eqnarray}
for all $i$.
\end{theorem}
\begin{proof}
The equation (\ref{eqn_par_impar}) follows from theorem \ref{theo_val_f}. 
\end{proof}

\begin{conjecture}
\begin{equation}
\cfauto{n}{n-2}_{(2,a)}=a^{2}(2^{n}-n-1)\label{eqn_2_n_n-2}    
\end{equation}
\end{conjecture}

The sequence
\begin{align*}
    2^{n}-n-1&=(0, 0, 1, 4, 11, 26, 57, 120, 247, 502, 1013, 2036, 4083, 8178, 16369,\nonumber\\
    &32752, 65519, 131054, 262125, 524268, 1048555, 2097130,\ldots)
\end{align*}
is known as Eulerian numbers (\seqnum{A000295} in OEIS).

\subsection{Case k=3}
When $k=3$ we obtain the equation
\begin{equation}\label{eqn_k_3}
    y^{(3)}=ae^{y}
\end{equation}
Solving (\ref{eqn_k_3}) with initial conditions $(0,0,x)$ and $a=-1$ we get the solution of Blasius equation
\begin{equation}
    u^{(3)}+u^{\prime\prime}u=0.
\end{equation}
The Blasius equation \cite{Blasius} describes the velocity profile of the fluid in the boundary layer which forms when fluid flows along a flat plate. Using the theorem \ref{theo_val_f} and the corollary \ref{cor_val_f} we reach the following result on Blasius numbers (\seqnum{A018893} in OEIS)
\begin{theorem}
Let
\begin{equation}
    (\mathrm{b}_{n})_{n\geq1}=(1,1,11,375,27.897,\ldots)
\end{equation}
denote the sequence of Blasius numbers. Then
\begin{enumerate}
\item $\mathrm{b}_{n}=B_{n}(0,0,\mathrm{b}_{1},\ldots,0,0,\mathrm{b}_{n-1})$, $n\geq2$.
\item $\mathrm{b}_{n+2}=\sum_{i=0}^{n}\binom{3n+2}{3i+2}\mathrm{b}_{n-i+1}\mathrm{b}_{i+1}$, $n\geq0$.
\end{enumerate}
\end{theorem}

On the other hand, the autonomous polynomials for the equation (\ref{eqn_k_3}) are
\begin{align*}
A_{3}^{(3)}(x,a)&=a,\\
A_{4}^{(3)}(x,a)&=ax,\\
A_{5}^{(3)}(x,a)&=a(x+x^{2}),\\
A_{6}^{(3)}(x,a)&=a(a+3x^{2}+x^{3}),\\
A_{7}^{(3)}(x,a)&=a(5ax+3x^{2}+6x^{3}+x^{4}),\\
A_{8}^{(3)}(x,a)&=a(11ax+16ax^{2}+15x^{3}+10x^{4}+x^{5}),\\
A_{9}^{(3)}(x,a)&=a(11a^{2}+84ax^{2}+(42a+15)x^{3}+45x^{4}+15x^{5}+x^{6}),\\
A_{10}^{(3)}(x,a)&=a(117a^{2}x+129ax^{2}+384ax^{3}+(99a+105)x^{4}+105x^{5}+21x^{6}+x^{7})
\end{align*}

and from here we obtain the following table with the first $(3,a)$-autonomous coefficients
\begin{table}[H]
    \centering
    \begin{tabular}{l|c|c|c|c|c|c|c|c}
         \backslashbox{$n$}{$i$}& 0&1 & 2& 3& 4 &5 &6&7\\\hline
         0& $a$&  &  &  & & &\\\hline
         1& 0&  $a$&  &  & & &\\\hline
         2& 0&  $a$&  $a$&  & & &\\\hline
         3& $a^2$&  0&  $3a$&  $a$&  & &\\\hline
         4& 0&  $5a^2$&  $3a$&  $6a$& $a$ & &\\\hline
         5& 0& $11a^2$&  $16a^2$&  $15a$&  $10a$& $a$ &\\\hline
         6& $11a^{3}$& 0&  $84a^2$&  $42a^2+15a$&  $45a$& $15a$ &$a$\\\hline
         7& 0&$117a^{3}$& $129a^{2}$&  $384a^2$&  $99a^2+105a$&  $105a$& $21a$ &$a$
    \end{tabular}
    \caption{$(3,a)$-autonomous coefficients}
    \label{tab:my_label}
\end{table}

\begin{theorem}
Some values of $(3,a)$-autonomous coefficients are
\begin{align*}
    \cfauto{n}{n}_{(3,a)}&=a,\\
    \cfauto{n}{n-1}_{(3,a)}&=a\binom{n}{2}.
\end{align*}
\end{theorem}
\begin{proof}
The results follow from theorem \ref{theo_val} with $l=0,1$ and by keeping in mind that $\stirlingsk{n}{n-1}=\binom{n}{2}$.
\end{proof}

\begin{conjecture}
\begin{equation}
\cfauto{n}{n-2}_{(3,a)}=a\binom{\binom{n}{2}}{2}    
\end{equation}
The numbers $\cfauto{n}{n-2}_{(3,1)}$ are the triangular numbers
\begin{equation}
    (0,0,3,15,45,105,210, 378, 630, 990, 1485,\ldots)
\end{equation}
(\seqnum{A050534} in OEIS).
\end{conjecture}
Finally, by the equations (\ref{eqn_bell_signo}), (\ref{eqn_pol_auto}) and (\ref{eqn_conj1_1}) we have

\begin{theorem}
Let
\begin{equation}
    (\mathrm{e}_{n})_{n\geq1}=(A_{n}^{(3)}(1,1))_{n\geq1}=(1,1,1,1,2,5,15,53,213,\ldots).
\end{equation}
denote the number of shifts 3 places left under exponentiation (\seqnum{A007548} in OEIS). Then
\begin{enumerate}
\item $\mathrm{e}_{n+3}=B_{n}(\mathrm{e}_{1},\ldots,\mathrm{e}_{n})$, $n\geq1$.
\item $(-1)^{n}\mathrm{e}_{n+3}=B_{n}(-\mathrm{e}_{1},\mathrm{e}_{2},\ldots,(-1)^{n-1}\mathrm{e}_{n})$, $n\geq1$.
\item $\mathrm{e}_{n+3}=\sum_{i=1}^{n}\stirlingsk{n}{i}\mathrm{e}_{i}$, $n\geq1$.
\item $\mathrm{d}_{3n}=\mathrm{b}_{n}+\sum_{i=1}^{3n}\cfauto{3n}{i}_{(3,1)}$.
\item $\mathrm{d}_{3n+j}=\sum_{i=1}^{3n+j}\cfauto{3n+j}{i}_{(3,1)}$, $j=1,2$.
\end{enumerate}
\end{theorem}

\subsection{Case k=4}

The equation to be studied is
\begin{equation}\label{eqn_k_4}
    y^{(4)}=ae^{y}.
\end{equation}
This equation is not commonly studied in the literature. Here we show the relation of this equation with the number of shifts left 3 places under Stirling transform, and also the relation with the numbers $A_{n}^{(4)}(1)$.
A list of exponential autonomous polynomials of the equation (\ref{eqn_k_4}) is as follows:
\begin{align*}
q_{1}(y,z,w,a)&=y,\\
q_{2}(y,z,w,a)&=z,\\
q_{3}(y,z,w,a)&=w,\\
q_{4}(y,z,w,a)&=a,\\
q_{5}(y,z,w,a)&=ay,\\
q_{6}(y,z,w,a)&=a(z+y^{2}),\\
q_{7}(y,z,w,a)&=a(w+3yz+y^{3}),\\
q_{8}(y,z,w,a)&=a(a+3z^{2}+4yw+6y^{2}z+y^{4}).
\end{align*}

From the equation (\ref{eqn_recu_A}) we calculate the first numbers $A_{n}^{(4)}(1)$,
\begin{align*}
q_{4}(0,0,0,1)&=A_{1}^{(4)}(1)=1,\\
q_{8}(0,0,0,1)&=A_{2}^{(4)}(1)=\binom{3}{3}A_{1}^{(4)}A_{1}^{(4)}=1,\\
q_{12}(0,0,0,1)&=A_{3}^{(4)}(1)=\binom{7}{3}A_{2}^{(4)}A_{1}^{(4)}+\binom{7}{7}A_{1}^{(4)}A_{2}^{(4)}=35,\\
q_{16}(0,0,0,1)&=A_{4}^{(4)}(1)=\binom{11}{3}A_{3}^{(4)}A_{1}^{(4)}+\binom{11}{7}A_{2}^{(4)}A_{2}^{(4)}+\binom{11}{11}A_{1}^{(4)}A_{3}^{(4)}=6140.
\end{align*}
Following theorem \ref{theo_val_f}, corollary \ref{cor_val_f} and equation (\ref{eqn_bell_signo}) we have the following recurrence relations for the numbers $A_{n}^{(4)}(1)$

\begin{theorem}
Let $(\mathrm{c}_{n})_{n\geq1}=(A_{n}^{(4)}(1))_{n\geq1}=(1,1,35,6140,\ldots)$. Then
\begin{enumerate}
\item $\mathrm{c}_{n}=B_{n}(0,0,0,\mathrm{c}_{1},\ldots,0,0,0,\mathrm{c}_{n-1})$, $n\geq2$.
\item $\mathrm{c}_{n+2}=\sum_{i=0}^{n}\binom{4n+3}{4i+3}\mathrm{c}_{n-i+1}\mathrm{c}_{i+1}$, $n\geq0$.
\item $(-1)^{n}c_{n}=B_{n}(0,0,0,-\mathrm{c}_{1},\ldots,0,0,0,(-1)^{n}\mathrm{c}_{n-1})$
\end{enumerate}
\end{theorem}

The autonomous polynomials associated with the equation (\ref{eqn_k_4}) are
\begin{align*}
A_{1}^{(4)}(x,a)&=A_{2}^{(4)}(x,a)=A_{3}^{(4)}(x,a)=x,\\
A_{4}^{(4)}(x,a)&=a,\\
A_{5}^{(4)}(x,a)&=ax,\\
A_{6}^{(4)}(x,a)&=a(x+x^{2}),\\
A_{7}^{(4)}(x,a)&=a(x+3x^2+x^{3}),\\
A_{8}^{(4)}(x,a)&=a(a+7x^2+6x^{3}+x^{4}),\\
A_{9}^{(4)}(x,a)&=a(6ax+10x^2+25x^{3}+10x^{4}+x^5),\\
A_{10}^{(4)}(x,a)&=a(16ax+32ax^2+75x^{3}+65x^{4}+15x^5+x^6),\\
A_{11}^{(4)}(x,a)&=a(36ax+136ax^2+(64a+175)x^3+315x^{4}+140x^{5}+21x^{6}+x^7).
\end{align*}

We now derive recurrence relations of the numbers $A_{n}^{(4)}(1,1)$ using the equations (\ref{eqn_bell_signo}), (\ref{eqn_pol_auto}), and (\ref{eqn_conj1_1}).
\begin{theorem}
Suppose 
\begin{equation}
    (\mathrm{d}_{n})_{n\geq1}=(A_{n}^{(4)}(1,1))_{n\geq1}=(1,1,1,1,1,2,5,15,53, 222, 1115, 6698,\ldots)
\end{equation}
the number of shifts left 3 places under Stirling transform (\seqnum{A336020} in OEIS). Then
\begin{enumerate}
\item $\mathrm{d}_{n+4}=B_{n}(\mathrm{d}_{1},\ldots,\mathrm{d}_{n})$, $n\geq1$.
\item $(-1)^{n}\mathrm{d}_{n+4}=B_{n}(-\mathrm{d}_{1},\mathrm{d}_{2}...,(-1)^{n-1}\mathrm{d}_{n})$, $n\geq1$.
\item $\mathrm{d}_{n+4}=\sum_{i=1}^{n}\stirlingsk{n}{i}\mathrm{d}_{i}$, $n\geq1$.
\item $\mathrm{d}_{4n}=\mathrm{c}_{n}+\sum_{i=1}^{4n}\cfauto{4n}{i}_{(4,1)}$.
\item $\mathrm{d}_{4n+j}=\sum_{i=1}^{4n+j}\cfauto{4n+j}{i}_{(4,1)}$, $j=1,2,3$.
\end{enumerate}
\end{theorem}

The following is a table of the first $(4,a)$-autonomous coefficients
\begin{table}[H]
    \centering
    \begin{tabular}{l|c|c|c|c|c|c|c|c}
         \backslashbox{$n$}{$i$}& 0&1 & 2& 3& 4&5 &6 & 7  \\\hline
         0& $a$&  &  &  & & & \\\hline
         1& 0&  $a$&  &  &  & &\\\hline
         2& 0&  $a$&  $a$&  &  & &\\\hline
         3& 0&  $a$&  $3a$&  $a$&  & & \\\hline
         4& $a^2$&  0&  $7a$&  $6a$& $a$& & \\\hline
         5& 0&  $6a^2$&  $10a$&  $25a$& $10a$& $a$& \\\hline
         6& 0&$16a^2$&  $32a^2$&  $75a$&  $65a$& $15a$& $a$  \\\hline
         7& 0&$36a^2$&  $136a^2$&  $64a+175$&  $315a$& $140a$& $21a$& $a$\\
    \end{tabular}
    \caption{$(4,a)$-autonomous coefficients}
    \label{tab:my_label}
\end{table}

\begin{theorem}
Some values of $(4,a)$-autonomous coefficients are
\begin{align}
    \cfauto{n}{n}_{(4,a)}&=a\label{eqn_sti_4_n_n}\\
    \cfauto{n}{n-1}_{(4,a)}&=a\binom{n}{2}\label{eqn_sti_4_n_n-1}\\
    \cfauto{n}{n-2}_{(4,a)}&=a\stirlingsk{n+2}{n}\label{eqn_sti_4_n_n-2}
\end{align}
\end{theorem}
\begin{proof}
The equations (\ref{eqn_sti_4_n_n})-(\ref{eqn_sti_4_n_n-2}) arise from theorem \ref{theo_val} with $l=0,1,2$.
\end{proof}

\begin{conjecture}
\begin{eqnarray}
    \cfauto{n}{n-3}_{(4,a)}&=&\frac{5a}{2}(n-1)\binom{n}{5}, \ n\geq5\label{eqn_sti_4_n_n-3}
\end{eqnarray}
\end{conjecture}

The sequence
\begin{eqnarray*}
    \cfauto{n}{n-3}_{(4,1)}&=&(10, 75, 315, 980, 2520, 5670, 11550, 21780, 38610, 65065, \\
    &&105105,163800, 247520, 364140, 523260, 736440, 1017450,\\
    &&1382535, 1850695,2443980, 3187800, 4111250, 5247450,\ldots)
\end{eqnarray*}
counts the number of 8-sequences of $[1,n]$ with 2 contiguous pairs, (\seqnum{A027778} in the OEIS).

%\Addresses

\bigskip
\hrule
\bigskip

\noindent 2020 {\it Mathematics Subject Classification}:
Primary 34A34; Secondary 11B37, 11B68, 11B73, 11B83. 

\noindent \emph{Keywords: } Bell's polynomial, autonomous function, autonomous polynomial, autonomous coefficient, Stirling number.

\bigskip
\hrule
\bigskip

\noindent (Concerned with sequences
\seqnum{A000111},
\seqnum{A000295},
\seqnum{A000367},
\seqnum{A002105},
\seqnum{A002445},
\seqnum{A007548},
\seqnum{A018893}, 
\seqnum{A027778}, 
\seqnum{A050534}, and
\seqnum{A336020}.)

\end{document}